\documentclass[11pt]{amsart}
\usepackage[english]{babel}
\usepackage[latin1,utf8]{inputenc}
\usepackage{amssymb,geometry,color,xcolor,graphicx}
\usepackage{amsmath,amsthm}
\usepackage{booktabs}
\usepackage{caption}
\usepackage{datetime}
\usepackage{dsfont}
\usepackage{amstext}
\usepackage{graphicx}
\usepackage{fancyhdr}
\usepackage{latexsym}
\usepackage{mathrsfs}
\usepackage{cases}
\usepackage{mathtools}
\usepackage{bm}
\usepackage{pgfplots}\pgfplotsset{width=10cm,compat=1.9}
\usepackage{subcaption}
\usepackage{wrapfig}
\usepackage{xcolor,listings}
\usepackage[numbered]{matlab-prettifier}
\usepackage{hyperref}
\hypersetup{
    colorlinks=true, 
    linktoc=all,     
    linkcolor=blue,  
}
\numberwithin{equation}{section}

\newcommand{\udef}{\mathrel{\mathop:}=}

\newcommand{\R}{\mathbb{R}}

\newcommand{\N}{\mathbb{N}}

\newcommand{\de}{\mathrm{d}}
\usepackage{dsfont}

\newcommand{\norm}[1]{\left\| #1 \right\|}

\theoremstyle{plain}
\newtheorem{thm}{Theorem}[section]

\newtheorem{lem}[thm]{Lemma}

\newtheorem{exm}[thm]{Example}

\let\tilde\widetilde
\let\hat\widehat

\newcommand{\parder}[2]{\frac{\partial#1}{\partial#2}}
\newcommand{\der}[2]{\frac{\de#1}{\de#2}}

\newcommand{\xx}{\mathbf{x}}

\renewcommand{\ss}{\scriptstyle}
\def\svdots{\vbox{\baselineskip=1.5pt\lineskiplimit=0pt
	\kern1.5pt \hbox{$\ss .$}\hbox{$\ss .$}\hbox{$\ss .$}}}

\begin{document}

\title{A  Numerical Method for a Nonlocal Form of Richards' Equation Based on Peridynamic Theory}
\author[Berardi]{Marco Berardi}
\address{Istituto di Ricerca sulle Acque, Consiglio Nazionale delle Ricerche, Via F. de Blasio 5, 70132 Bari, Italy}
\email{marco.berardi@ba.irsa.cnr.it}
\author[Difonzo]{Fabio V. Difonzo}
\address{Dipartimento di Matematica, Universit\`a degli Studi di Bari Aldo Moro, Via E. Orabona 4, 70125 Bari, Italy}
\email{fabio.difonzo@uniba.it}
\author[Pellegrino]{Sabrina F. Pellegrino}
\address{Dipartimento di Management, Finanza e Tecnologia, Universit\`a LUM Giuseppe Degennaro, S.S. 100 Km 18, 70010 Casamassima,Italy}
\email{pellegrino@lum.it}




\begin{abstract}
Forecasting water content dynamics in heterogeneous porous media has significant interest in hydrological applications; in particular, the treatment of infiltration when in presence of  cracks and fractures can be accomplished resorting to peridynamic theory, which allows a proper modeling of non localities in space. In this framework, we make use of Chebyshev transform on the diffusive component of the equation and then we integrate forward in time using an explicit method. We prove that the proposed spectral numerical scheme provides a solution converging to the unique solution in some appropriate Sobolev space. We finally exemplify on several different soils, also considering a sink term representing the root water uptake.
\end{abstract}

\maketitle

\pagestyle{myheadings}
\thispagestyle{plain}
\markboth{M. BERARDI, F.V. DIFONZO, S.F. PELLEGRINO}{PERIDYNAMIC RICHARDS' EQUATION}

\maketitle

\section{Introduction}

Environmental protection and related sustainability management policies demand a thorough understanding of complex coupling between hydrology, soil sciences, ecology, agronomy, atmospheric sciences, calling for deeper mathematical modeling and numerical methods able to deal with the multiphysics processes involved in these environmental phenomena. In particular, flow processes in unsaturated media have to be studied for a better understanding of the whole water cycle; a correct managing of irrigation needs relies, for instance, on robust numerical solvers for unsaturated flows with root water uptake (see for instance, \cite{Difonzo_Masciopinto_Vurro_Berardi,Romashchenko_Irrigation_science_2021}), or it is the basis for forecasting contaminant transport in the vadose zone (see for instance \cite{Simunek_VanGenuchten_2016}). Classical local advection-diffusion equations in porous media often fail to describe accurately such complex phenomena.\\
The idea of incorporating non-local behaviors in standard unsaturated flow models is gaining interest in recent times. Besides non-localities in space, which are the focus of this paper, also non-local effects in time can be considered, that generally account for memory terms in the advection-diffusion equations: in some pioneering works in the early '60s
\cite{Rawlins_Gardner_1963}) it had been already noticed that diffusivity depends not only on water content, but also explicitly on time, and this argument has been then extended also to hydraulic conductivity (see \cite{Guerrini_Swartzendruber_SSSAJ_1992}; later on a model, in which derivative of water content on time is fractional, has been first proposed in \cite{Pachepsky_et_al_JoH_2003} and then generalized in \cite{Kavvas_et_al_HESS_2017}. A memory component has been observed also when modeling water stress in the root water uptake: the experimental evidence of such "ecological memory" of plant roots has been noticed, for instance, in \cite{Wu_BenGal_et_alAGWAT_2020,Carminati_VZJ_2012} and has been recently formalized in \cite{Berardi_Girardi_memory_2022}.
\medskip

When dealing with spatial discontinuities or significant  heterogeneities, classical local formulations of flow and transport phenomena present severe limitations; for instance, in some cases, standard unsaturated flow models can not forecast correctly water dynamics; as reported in \cite{neuweilerEtAl2012}, when modeling fast infiltration processes (for instance infiltration after a heavy rainfall event), "first arrival time at the groundwater [...] are often underpredicted" because of preferential flow paths. These preferential flows can be ascribed to non-equilibrium of water pressure at a local scale. As a matter of facts, there is an experimental evidence that pore structures in natural soils dynamically change due to alternating swelling and shrinkage processes (see for instance \cite{Coppola_et_al_VZJ_2015}): this phenomenon can be described by a dual permeability approach, by which the bulk porous medium consists of two dynamic interacting pore domains: (i) the fracture (from shrinkage) pore domain and (ii) the aggregate (interparticles plus structural pores), respectively (see \cite{Coppola_et_al_WRR_2012}): in practice, two different unsaturated flow equations are considered in each part of this domain. Analogously, in the context of solute transport, the solute exchange between mobile and immobile water has been modeled by a delay term in \cite{Masciopinto_Passarella_2018}, and, in a computational framework, this approach has been  implemented in \cite{Municchi_DiPasquale_Dentz_Icardi_CPC_2021}.
\bigskip

More in detail, multirate mass transfer is modeled assuming advection-diffusion on the fast mobile continuum and only diffusion in the slow immobile continuum: after solving analytically the diffusion model, the consequent fast domain model results non-local in time (\cite{Carrera_et_al_1998}). In this dual-continuum framework, the pioneering work \cite{neuweilerEtAl2012} shows that the linearization of the nonlinear diffusion equation, governing capillary flow in the slow continuum, ensures a good description of the averaged water content dynamics in the slow domain: therefore, they derive a non-local Richards' equation in the mobile domain, endowed with a memory kernel encoding mass transfer dynamics of the slow domains. \\
From the viewpoint of applications, in this context, dessication cracks impact the efficiency of irrigation and provokes
a fast leakage of nutrients and contaminants below the root zone into the groundwater. Even worst is the effect of such cracks into an earth dam, since it may lead to the failure of the dam itself. \\
On the other hand, several efforts have been accomplished towards a comprehensive modeling and efficient numerical solvers for such nonlocal problems. In \cite{yan2020} the coupling  of peridynamic formulation of chemical
transport with water flow is proposed in the unsaturated context, and an implicit numerical solver is implemented, and tested over different case studies, in order to show the ability of the model to recognize discontinuities and heterogeneities, including stationary cracks,
propagating cracks, and randomly distributed permeable and impermeable inclusions. In \cite{Delgoshaie_et_al_JoH_2015} authors discuss how a single continuum model can properly catch the contributions from
all the flow paths only if the control volume (i.e. the computational cell) is much larger than the longest connections between pores: therefore a non-local model is proposed therein, showing that if the longest connections are much smaller than the size of the control volume, these models converge to Darcy's law. A significant work has be presented in \cite{Ni_et_al_CMAME_2020}, in which the peridynamic theory is employed for simulating hydraulic fracture propagation in saturated porous media, and it is coupled with FEM for forecasting fluid flow therein. In this paper we aim at introducing a tailored numerical method for the corresponding peridynamic model of Richards' equation describing the unsaturated flow; for the sake of clarity we should say that peridynamic theory was introduced by Silling in~\cite{SILLING2000} as a nonlocal version of elasticity theory, for modeling long-range interactions occurring in real materials, ruling several phenomena like fractures, instabilities and cracks. In general, peridynamic models consist of an integro-differential equation not involving spatial derivatives and describe the motion of a material body subjected to external loading conditions.
The theory prescribes the existence of a domain influence, called horizon, which represents a measure of the nonlocality of the model and defines the range of interactions between material particles.

 In this framework, the remaining of the paper is structured as follows. In Section \ref{sec:periModel} an introduction to nonlocal framework and a peridynamic formulation of Richards' equation is given, with all the necessary assumptions to justify the current setting. Then, in Section \ref{sec:numericalMethod}, we propose a numerical method to integrate forward in time a semi-discretized version of the equation, leveraging spectral theory and Chebyshev transform properties to prove convergence results of the discretized solution to the exact one. The implementation of Chebyshev collocation method provides a good accuracy and does not require to impose periodic boundary conditions. Finally, in Section \ref{sec:numericalSimulations} we exemplify on different soils with several type of Dirichlet boundary conditions to support our findings.

\subsection{A short overview on Richards' equation}

It is well known that Richards' equation is a mass conservation law in terms of the volumetric moisture content $\theta$ and of the soil matric head $h_m$ defined on some compact domain $\Omega\subseteq\R^3$, coupled with the Buckingham-Darcy's law for the description on the flux:
\begin{align*}
\parder{\theta}{t}(\xx,t) &= -\nabla q(\xx,t)+S(\xx,\theta),\quad\xx\in\Omega \\
q(\xx,t) &= -K(h_m)\nabla(h_m-z),
\end{align*}
where $z$ is the elevation component of the space variable $\xx$,  $\theta$ represents the volumetric water content, $K$ is the so called hydraulic conductivity and $S(\xx,\theta)$ is a source or sink term describing, for instance, the root water uptake.
Thus, Richards' equation reads as
\begin{equation}\label{eq:Richards}
\parder{\theta}{t}(\xx,t)-\nabla\left(K(h_m)\nabla(h_m-z)\right)=S(\xx,\theta),\quad \xx\in\Omega,\,t\in[0,T],
\end{equation}
endowed with suitable initial and boundary conditions.

With the hypothesis that air pressure in the pores is constant, Richards' equation assumes that matric head at a given location is in equilibrium and that there exists a bijective function relating $\theta$ with $h_m$, called water retention curve (see \cite{neuweilerEtAl2012}), which is generally defined according to empirical functions. Moreover, for Richards' equation to be well posed, $K$ must be smooth enough to guarantee existence and uniqueness of solutions, also in case of heterogeneous soils with smooth boundary (see \cite{berardiDifonzoEFM2020} and references therein). In particular, hereafter and through the whole paper, $K$ and $h_m$ will be assumed to be locally Lipschitz on their respective domains.  \\
However, in case of desiccation cracks or anisotropic soils could affect well-posedness of Richards' equation \eqref{eq:Richards} and prevent existence of any solution. An alternative approach has been proposed in \cite{JabakhanjiMohtar2015}, where theory of elasticity for solid mechanics has been applied to unsaturated, heterogeneous, anisotropic soils. In this case, though, the flow density function depends on the position, matric head or moisture content, instead of the relative distance and relative displacement \cite{SILLING2000,SILLING2010}.
\\
The numerous numerical issues arising when solving Richards' equation in a computational framework rely mainly in its nature of highly nonlinear degenerate elliptic parabolic PDE. Here we just mention some significant references for the main numerical problems arising in Richards' equation. For instance, since implicit methods are generally used for time integration, the arising nonlinear problems have been studied with different methods, such as Newton's (e.g. \cite{Casulli_Zanolli_SIAM_2010,Bergamaschi_Putti}, Picard (\cite{Celia_et_al}), L-Scheme or its variants (\cite{Pop_JCAM_2004,Mitra_Pop_CAMWA_2019}). Even richer is the literature on spatial discretization techniques, for which we refer to \cite{Arbogast_Wheeler_1996,Li_Farthing_Miller_ADWR_2007,Manzini_Ferraris_ADWR_2004,Kees_Farthing_Dawson_CMAME_2008} and references therein. As regards numerical integration over layered discontinuous geological formations, a domain decomposition approach is followed in \cite{Seus_Mitra_Pop_Radu_Rohde_2018}, while a transversal method of lines is adopted in \cite{Berardi_Difonzo_Vurro_Lopez_ADWR_2018,Berardi_Difonzo_Lopez_CAMWA_2020}. \\
In this paper, we are looking at the 1D version of \eqref{eq:Richards} equipped with initial and Dirichlet boundary conditions, in which diffusion evolves exclusively along the depth, so that $\Omega=[0,Z]$ for some $Z>0$, and the forcing term $S$ only depends on $z\in[0,Z]$. Thus, one considers
\begin{equation}\label{eq:Richards1D_IBVC}
\parder{\theta}{t}(z,t)-\parder{}{z}\left(K(h_m)\parder{}{z}(h_m-z)\right) =
S(z),\quad z\in[0,Z],\,t\in[0,T].
\end{equation}

\section{Peridynamic Model: assumptions and derivation}\label{sec:periModel}

Let us consider a compact domain $\Omega$ with smooth boundary and let us define
\begin{equation}\label{eq:horizon}
B_\delta(z)\udef\{z'\in\Omega\,:\,\|z'-z\|\leq\delta\},
\end{equation}
the \emph{horizon} of $z$ of radius $\delta>0$. We assume that moisture dynamics at $z$ is only affected by pairwise interaction with $z'\in B_\delta(z)$; points outside the horizon of $z$ do not contribute to any dynamics therein.

The model is built on the concept of peripipes. Given any $z\in\Omega$, we assume that for each $z'\in B_\delta(z)$ there exists a fictitious pipe, called \emph{peripipe}, connecting every $z$ to $z'$.
We assume that the following requirements hold for any peripipe (see \cite{jabakhanjiphd2013}):
\begin{enumerate}
    \item Moisture is stored at the endpoints $z,x'$ of a peripipe, and zero moisture content is located along a peripipe;
    \item moisture flows in the direction of the peripipe and no transversal flux crosses its boundaries
    \item a peripipe is purely resistive, it has zero reactance and its response is proportional to $H(z)-H(z')$;
    \item a peripipe has uniform conductivity;
    \item peripipe conductivity is function of medium conductivity at its endpoints;
    \item the length of a peripipe is $\|z-z'\|$;
    \item peripipe response may also depend on the its length.
\end{enumerate}

Following \cite{JabakhanjiMohtar2015} and requirements above, we assume that the rate of volumetric moisture flow from a point $z'$ to a point $z$ per unit volume of $z$ and per unit volume of $z'$ is given by
\begin{equation}\label{eq:J}
J(z,z',t)=C(z,z')(H(z',t)-H(z,t)),
\end{equation}
where $C(z,z')$ is the peridynamic hydraulic conductance density and $H(z,t)$ is the total hydraulic potential, defined as
\[
H(z,t)=h_m(z,t)+z.
\]
Hereafter, for the sake of readability, we omit time dependence, unless required by the context.

The peripipe conductance depends on the peridynamic hydraulic conductivity $\kappa(z',z)$, which is an intrinsic material property (related to the classical hydraulic conductivity $K$), in the following way:
\begin{equation}\label{eq:C}
C(z,z')=\frac{\kappa(z',z)}{\|z-z'\|},
\end{equation}
where
\begin{equation}\label{eq:kappa}
\kappa(z,z')\udef K\varphi(z-z').
\end{equation}
The function $\varphi(z-z')$ is the so-called \emph{influence function}, representing a convolution kernel relating the horizon \eqref{eq:horizon} with the nature of boundary conditions assigned to \eqref{eq:Richards1D_IBVC}. The shape of such an even function and the way to select it turns out to be crucial, as we will see in Section \ref{sec:influenceFunction}.

Therefore, the changes of moisture stored at $z$ and at $z'$, mediated by the peripipe $zz'$, are given by
\begin{align*}
\Delta V_m(z,z') &= \kappa(z',z) \frac{H(z')-H(z)}{\|z-z'\|}\de V_{z'}\de V_z, \\
\Delta V_m(z',z) &= \kappa(z,z') \frac{H(z)-H(z')}{\|z'-z\|}\de V_{z}\de V_z'.
\end{align*}
As an immediate consequence it must hold $\kappa(z,z')=\kappa(z',z)$. \\
In case of inhomogeneous soils in unsaturated regime, above relations could be leveraged to define a peridynamic conductivity density by setting
\begin{equation}\label{eq:peryKappa}
\kappa(z,z')\udef\frac{\kappa(z)+\kappa(z')}{2},
\end{equation}
where $\kappa(z)\equiv\kappa(z,0)$, as proposed in \cite{yan2020,JabakhanjiMohtar2015}. \\
Now, since the change over time of volumetric moisture content due to $z'$ at time $t$, on the account of \eqref{eq:J}, is given by
\[
\parder{\theta}{t}(z|z',t)=J(z,z'),
\]
from which
\[
\parder{\theta}{t}(z,t)=\int_{B_\delta(z)}J(z,z')\,\de V_{z'}+S(z).
\]
Thus, using \eqref{eq:C} and with peridynamic conductivity given by \eqref{eq:peryKappa}, our model \eqref{eq:Richards1D_IBVC}, endowed with Dirichlet boundary conditions, reads
\begin{subequations}\label{eq:peryRichards2}
\begin{align}
\parder{\theta}{t} &= \int_{B_\delta(z)}\frac{\varphi(z'-z)}{\|z'-z\|}\frac{K(z)+K(z')}{2}[H(z')-H(z)]\,\de V_{z'}+S(z), \label{eq:peryRichards21D} \\
\theta(z,0) &= \theta^{0}(z),\quad z\in[0,Z], \label{eq:peryRichards21D_IC_t=0} \\
\theta(0,t) &= \theta_{0}(t),\quad t\in[0,T], \label{eq:peryRichards21D_DirichletBC_z=0} \\
\theta(Z,t) &= \theta_{Z}(t),\quad t\in[0,T]. \label{eq:peryRichards21D_DirichletBC_z=Z}
\end{align}
\end{subequations}

\subsection{Selection of the influence function}\label{sec:influenceFunction}

Usually (see, e.g. \cite{yan2020,Bobaru_Duangpanya_JCP_2012,JabakhanjiMohtar2015})
$\varphi(z)$ in \eqref{eq:kappa} represents a convolution kernel, which can be chosen as a uniform influence function
\[
\varphi(z)\udef
\begin{cases}
\frac{2}{\delta}, & \|z\|\leq\delta, \\
0, & \|z\|>\delta,
\end{cases}
\]
or as a linear influence function
\[
\varphi(z)\udef
\begin{cases}
1-\frac{\|z\|}{\delta}, & \|z\|\leq\delta, \\
0, & \|z\|>\delta.
\end{cases}
\]
However, since such kernels would suggest the model to weigh more those cells where they are nonzero, and since our boundary conditions would typically be of Dirichlet type, we propose to consider a \emph{distributed influence function} (see Figure \ref{fig:influenceFunction}), concentrated on the domain boundary, of the form
\begin{equation}\label{eq:distributedInfluenceFunction}
\varphi(z)\udef
\begin{cases}
\frac{\|z\|-1+\delta}{\delta}, & \|z\|\geq1-\delta, \\
0, & \|z\|<1-\delta.
\end{cases}
\end{equation}

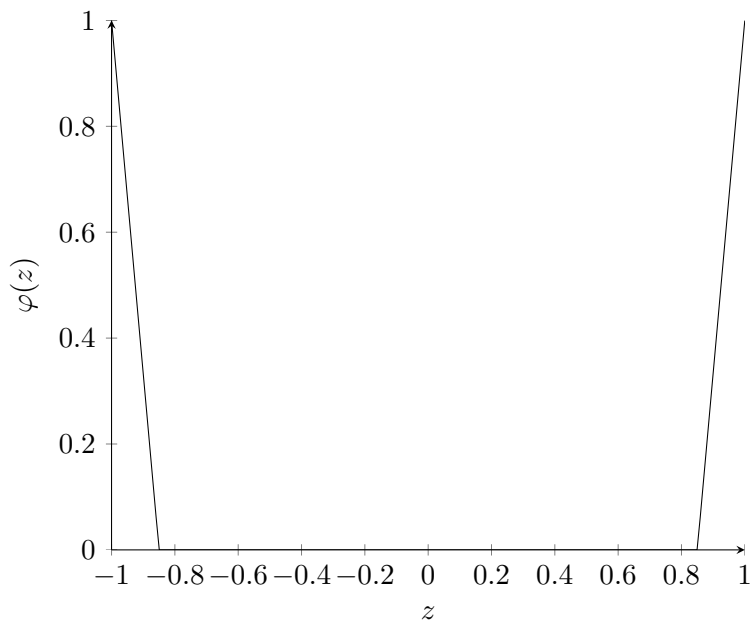
\begin{figure}
\centering
\begin{tikzpicture}
\begin{axis}[
    axis lines = left,
    xlabel = \(z\),
    ylabel = {\(\varphi(z)\)},
]

\addplot [
    domain=-1:1,
    samples=200,
    color=black,
]
{max(0,abs((x)/0.15)-(1-0.15)/0.15)};
\end{axis}
\end{tikzpicture}
\caption{Distributed influence function
defined in \eqref{fig:influenceFunction} with $\delta=0.15$.}
\label{fig:influenceFunction}
\end{figure}

In so doing, we are suggesting the model to averaging out not just what happens in the middle of the dynamics, but rather the behavior around each point of the spatial domain. In all our experiments, presented in Section \ref{sec:numericalSimulations}, uniform and linear influence functions do not make our proposed numerical method converge, resulting in instabilities and blow-ups after a relatively small amount of time integration; on the other hand, and as presented below, using \eqref{eq:distributedInfluenceFunction} guarantees stability and convergence, plus a reasonable shape of the numerical solutions.

\section{Numerical Method}\label{sec:numericalMethod}

The nonlocal Richards' equation~\eqref{eq:peryRichards2} can be discretized in space by using Chebyshev polynomials. This approach is typically used when the integral operator can be expressed in terms of convolution products~\cite{LPcheby,LPcheby2022,LPeigenv}. Moreover, the choice of such kind of polynomials allows us to overcome the limitation of imposing periodic boundary condition, which is necessary when dealing with Fourier trigonometric polynomials.

The proposed technique consists in looking for an approximation of $\theta(x,t)$ in the form of a finite linear combination of Chebyshev polynomials of the first kind. To do so, we can assume the spatial domain to be $[-1,1]$, as we can benefit of the orthogonality properties of the polynomials. However, a more general interval can be used as spatial domain by linearity. Moreover, for time integration we use the explicit Euler method, as in \cite{yan2020}.

In this section we briefly make a review on Chebyshev polynomials, then we derive the semidiscrete model of~\eqref{eq:peryRichards2} and finally prove the convergence of the proposed method.

\subsection{Basic overview on Chebyshev polynomials}

Chebyshev polynomials of the first kind, $T_k(z)$ are defined by
\[
    T_k(z)=\cos\left(k\arccos(z)\right),\qquad z\in[-1,1],\,\, k\in\N,
\]
and are orthogonal with respect to the weight function $w(z)\udef\left(\sqrt{1-z^2}\right)^{-1}$.

These polynomials are commonly used in the context of spectral approximation because they satisfy an interpolation property: given an integer $N$, any sufficiently smooth function $u$ defined on $[-1,1]$ can be expanded as an $(N+1)$-term linear combination of polynomials $u^N$ given by
\begin{equation}\label{eq:ifct}
u^N(z)\udef\sum_{k=0}^N \tilde{u}_k\ T_k(z),
\end{equation}
where $\tilde{u}_k$ are the coefficients of the expansion and approximate the Chebyshev coefficients
\[
\hat{u}_k=\frac{2}{\pi c_k}\int_{-1}^1 u(x)T_k(z)w(x)\, \de z,
\]
with
\[
c_k=\begin{cases}
2&\quad k=0\\
1&\quad k\ne0.
\end{cases}
\]

The explicit expression of $\tilde{u}_k$ depends on the choice of the grid points used to discretize $[-1,1]$. In particular, if we choose the Gauss-Lobatto collocation points
\begin{equation}
\label{eq:GL}
z_h\udef\cos\left(\frac{h\pi}{N}\right)  ,\qquad h=0,\dots,N,
\end{equation}
the expression of $\tilde{u}_k$ is
\begin{equation}
\label{eq:fct}
\tilde{u}_k=\frac{1}{\gamma_k}\sum_{h=0}^N u(x_h)\ T_k(z_h) w_h,
\end{equation}
where $\gamma_k$ is a normalization constant defined by
\begin{equation}
\label{eq:gamma}
\gamma_k\udef\begin{cases}
\pi\quad& k=0,N\\
\frac{\pi}{2}\quad& k=1,\dots,N-1
\end{cases}
\end{equation}
and
\begin{equation}
\label{eq:wh}
    w_h\udef\begin{cases}
    \frac{\pi}{2N}\quad& h=0,N\\
    \frac{\pi}{N}\quad& h=1,\dots,N-1.
    \end{cases}
\end{equation}

Equation~\eqref{eq:ifct} represents the inverse discrete Chebyshev transform, while the coefficients $\tilde{u}_k$ in~\eqref{eq:fct} correspond to the discrete Chebyshev transform. They can be efficiently computed using the Fast Fourier Transform. Additionally, they fulfill the same properties of the Fourier transform. In particular, we can rewrite a convolution product in the physic space as a multiplication of the Chebyshev transform of each factor in the frequency space.

The following result shows the rate of convergence of the Chebyshev approximation.
\begin{thm}[see~\cite{Canuto}]
\label{eq:th1}
For any $0\le\mu\le 2$ and $u\in L^2([-1,1])$, there exists a positive constant $C$ independent on $N$, such that
\begin{equation}
\label{eq:convrate}
\norm{u-u^N}_{L^2([-1,1])} \le \frac{C}{N^{2-\mu}} \norm{u}_{L^2([-1,1])}.
\end{equation}
\end{thm}

In the next section, to lighten the notation, we denote the Chebyshev transform by $\mathcal{F}$ and the inverse Chebyshev transform by $\mathcal{F}^{-1}$.

\subsection{Chebyshev semi-discrete collocation method for the nonlocal Richards' equation}

In what follows, we develop a spectral approximation of~\eqref{eq:peryRichards2} by using the Chebyshev transform. We fix $N>0$ and assume $\Omega=[-1,1]$. We can discretize the spatial domain by the Guass-Lobatto points $z_h$, $h=0,\dots,N$ defined in~\eqref{eq:GL}.

If we set
\begin{align*}
\Lambda(z) &\udef K(z)H(z), \\
\overline{\varphi}(z) &\udef \frac{\varphi(z)}{\norm{z}}, \\
\beta &\udef \int_{-1}^{1}\overline{\varphi}(z)\,\de z,
\end{align*}
then, since from distributed influence function definition \eqref{eq:distributedInfluenceFunction} it follows that $\beta=2\left(1+\frac{1-\delta}{\delta}\ln(1-\delta)\right)$, we can rewrite model \eqref{eq:peryRichards21D} as
\begin{equation}\label{eq:periRichard3}
\begin{split}
\parder{\theta}{t}
&=\int_{B_\delta(z)}\frac{\varphi(z'-z)}{\|z'-z\|}\  \frac{K(z')+K(z)}{2}\ [H(z')-H(z)]\,\de V_{z'}+S(z),\\
&=\frac12\left[\left(\overline{\varphi}\ast\Lambda\right)(z) + K(z)\left(\overline{\varphi}\ast H\right)(z)-H(z)\left(\overline{\varphi}\ast K\right)(z)-\beta\Lambda(z)\right] + S(z).
\end{split}
\end{equation}

Thus, the right hand side of~\eqref{eq:periRichard3} can be computed by means of the finite discrete Chebyshev transform. Indeed, we have
\begin{align}
\label{eq:firstterm}
\left(\overline{\varphi}\ast \Lambda\right)(z)&=\mathcal{F}^{-1}\left(\mathcal{F}\left(\overline{\varphi}\right)\mathcal{F}\left(\Lambda\right)\right)(z),\\
\label{eq:secondterm}
\left(\overline{\varphi}\ast H \right)(z)&=\mathcal{F}^{-1}\left(\mathcal{F}\left(\overline{\varphi}\right)\mathcal{F}\left(H\right)\right)(z),\\
\label{eq:thirdterm}
\left(\overline{\varphi}\ast K \right)(z)&=\mathcal{F}^{-1}\left(\mathcal{F}\left(\overline{\varphi}\right)\mathcal{F}\left(K\right)\right)(z).
\end{align}
So, at each collocation point $z_h$, the semi-discretization of the model reads
\begin{equation}
\label{eq:scheme}
\begin{split}
\parder{\theta}{t}(z_h,t)&=\frac12\left(\mathcal{F}^{-1}\left(\mathcal{F}\left(\overline{\varphi}\right)\mathcal{F}\left(\Lambda\right)\right)(z_h)+K(z_h)\ \mathcal{F}^{-1}\left(\mathcal{F}\left(\overline{\varphi}\right)\mathcal{F}\left(H\right)\right)(z_h)\right)\\
&\quad-\frac12\left(H(z_h)\ \mathcal{F}^{-1}\left(\mathcal{F}\left(\overline{\varphi}\right)\mathcal{F}\left(K\right)\right)(z_h)+\beta\Lambda(z_h)\right)+ S(z_h)
\end{split}
\end{equation}

The function $\Lambda$ is defined as the product between the conductivity $K$ and the hydraulic potential $H$: therefore, to compute its Chebyshev transform, we first need to compute a product. The computational cost to obtain this term could be efficiently reduced by observing that the Chebyshev coefficients of $\Lambda$ can be obtained from the Chebyshev coefficients of $H$ and $K$.

Indeed, the following result holds (see~\cite{Baszenski19971}).
\begin{thm}
\label{th:lambda}
Let $N\in\N$. If $H$ and $K$ are approximated by a finite series of Chebyshev polynomials $H^N$ and $K^N$, respectively, given by
\begin{equation*}
    H^N(z)=\sum_{j=0}^N \tilde{H}_k T_k(z),\quad K^N(z)=\sum_{j=0}^N \tilde{K}_j T_j(z),
\end{equation*}
then the product $\Lambda(z)=H(z)K(z)$ can be approximated by the following $2N+1$ combination of Chebyshev polynomials
\begin{equation*}
    \Lambda^N(z)=\sum_{j=0}^{2N}\tilde{\Lambda}_j T_j(z),
\end{equation*}
where the coefficients $\tilde{\Lambda}_j$ are given by
\[
   2\tilde{\Lambda}_j=\begin{cases}
   2\tilde{H}_0\tilde{K}_0+\sum_{\ell=1}^N \tilde{H}_\ell\tilde{K}_\ell,\quad&j=0\\
   \sum_{\ell=0}^j \tilde{H}_{j-\ell}\tilde{K}_\ell+\sum_{\ell=0}^{N-j}\tilde{H}_{j+\ell}\tilde{K}_\ell + \sum_{\ell=j}^N \tilde{H}_{\ell-j}\tilde{K}_{\ell},\quad&j=1,\dots,N\\
   \sum_{\ell=j-N}^N \tilde{H}_{j-\ell}\tilde{K}_\ell,\quad&j=N+1,\dots,2N.
   \end{cases}
\]
\end{thm}

The application of Theorem~\ref{th:lambda} implies that the first term in the right hand side of~\eqref{eq:scheme} is discretized by $2N+1$ mesh points. Therefore, to maintain the consistency of the scheme, the discretization of the remaining terms on the right hand side of~\eqref{eq:scheme} is accomplished by considering $2N+1$ Gauss-Lobatto collocation points.

\subsection{Convergence of the semi-discrete scheme}
We prove the convergence of the spectral semi-discrete method in a suitable weighted Hilbert space. Throughout this section, $C$ denotes a generic constant.

We consider the weighted Lebesgue space
\[
L^2_w([-1,1])=\left\{u\in L^2\,:\,\int_{-1}^1 u^2(z) w(z) \de z <+\infty \right\}
\]
equipped with the inner product and the norm respectively
\[
(u,v)_w = \int_{-1}^1 u(z)v(z)w(z)\,\de z,\qquad \norm{u}^2_w = (u,u)_w,
\]
where $w(z)=\left(\sqrt{1-z^2}\right)^{-1}$.

For any $s\ge0$, we set
\[
H^s_w([-1,1])=\left\{u\in L^2_w([-1,1])\ | \ \norm{u}_{s,w}<+\infty\right\},
\]
where
\[
\norm{u}_{s,w}^2=\sum_{|\alpha\le s|}\norm{D^\alpha u}_{w}^2.
\]
Let $S_{N}$ be the space of Chebyshev polynomials of degree $N$,
\[
S_{N}\udef \text{span}\left\{T_h(x)\ |\ 0 \le h\le N\right\}\subset L^2_w([-1,1]),
\]
and $P_N: L^2_w([-1,1]) \to S_N$ be an orthogonal projection operator
\[
P_Nu(x)\udef \sum_{h=0}^N \hat{u}_h T_h(x)w_h,
\]
for $w_h$ defined in~\eqref{eq:wh}, such that for any $u\in L^2_w([-1,1])$, the following equality holds
\begin{equation}
\label{eq:orthogonal}
(u-P_Nu,\varphi)_w = 0,\quad\text{for every $\varphi\in S_N$}.
\end{equation}
The projection operator $P_N$ commutes with derivatives in the distributional sense:
\[
\partial_t^q P_Nu = P_N\partial_t^q u,\quad q\in\N,q\geq1,
\]
where, as usual, $\partial_t\udef\parder{}{t}$. \\
Letting $s\geq1$, we denote by $X_s \udef \mathcal{C}^0\left(0,T; H^s_w([-1,1])\right)$ the space of all continuous functions in the weighted Sobolev space $H^s_w([-1,1])$,
with norm
\[
\norm{u}_{X_s}^2 \udef \max_{t\in[0,T]}\norm{u(\cdot,t)}_{s,w}^2,
\]

for any $T> 0$.
We denote by $\mathcal{L}$ the nonlocal integral operator of~\eqref{eq:peryRichards2}, namely
\begin{equation}
\label{eq:L}
\mathcal{L}\left(\theta\right)\udef\int_{B_\delta(z)}\frac{\varphi(z'-z)}{\|z'-z\|}\frac{K(z)+K(z')}{2}[H(z')-H(z)]\,\de V_{z'}.
\end{equation}
Then, the semi-discrete spectral scheme for~\eqref{eq:peryRichards2} can be rewritten as
\begin{gather}
\label{eq:schemePN}
\parder{\theta^N}{t} = P_N \mathcal{L}(\theta^N) + P_N S(z),\\
\label{eq:initial_scheme}
\theta^N(z,0) = P_N \theta^0(z),
\end{gather}
where $\theta^{N}(z,t)\in S_N$ for every $0\le t\le T$.

To obtain the convergence of the semi-discrete scheme, we need of the following lemma.
\begin{lem}[{\cite[Theorem 3.1]{Canuto}}]
\label{lm:sobolev}
For any real $0\le \mu\le s$, there exists a positive constant $C$ such that
\begin{equation}
\label{eq:sobolev}
\norm{u-P_N u}_{H^\mu_{s,w}{\mu}([-1,1])} \le \frac{C}{N^{s-\mu}}\norm{\theta}_{H^s_{w}([-1,1])}, \quad\text{for every $\theta\in H^{s}_w([-1,1])$}.
\end{equation}
\end{lem}

Recalling that $K$ and $h_m$ are locally Lipschitz in their respective domains, we can prove the following theorem.
\begin{thm}
\label{th:convergence}
Let $s\ge 1$ and $\theta(z,t)\in X_s$ be the solution to the initial-boundary-valued problem~\eqref{eq:peryRichards2} and $\theta^N(z,t)$ be the solution to the semi-discrete scheme~\eqref{eq:schemePN}-\eqref{eq:initial_scheme}.
Then, there exists a positive constant $C$, independent on $N$, such that
\begin{equation}
\label{eq:order_conv}
\norm{\theta-\theta^N}_{X_1} \le C(T) \left(\frac{1}{N}\right)^{s-1} \norm{u}_{X_s},
\end{equation}
for any initial data $\theta^0\in H^s_w([-1,1])$ and for any $T > 0$.
\end{thm}

\begin{proof}
Let $s\ge 1$. Using the triangular inequality, we have
\begin{equation}
\label{eq:triangular}
\norm{\theta-\theta^N}_{X_1} \le \norm{\theta-P_N \theta}_{X_1} + \norm{P_N \theta - \theta^N}_{X_1}.
\end{equation}
Lemma~\ref{lm:sobolev} implies
\[
\norm{(\theta-P_N \theta)(\cdot,t)}_{H^1_w([-1,1])} \le \frac{C}{ N^{s-1}}\norm{\theta(\cdot,t)}_{H^s_w([-1,1])}.
\]
Therefore,
\begin{equation}
\label{eq:1term}
\norm{\theta-P_N \theta}_{X_1} \le \frac{C}{ N^{s-1}}\norm{\theta}_{X_s}.
\end{equation}
Subtracting~\eqref{eq:schemePN} from~\eqref{eq:peryRichards2} and taking the weighted inner product with $P_N \theta - \theta^N \in S_N$, we have
\begin{equation}
\label{eq:difference}
\begin{split}
0=&\underbrace{\int_{-1}^1 \left(\partial_t\theta(z,t)-\partial_t\theta^N(z,t)\right)\left(P_N \theta(z,t) - \theta^N(z,t)\right)w(z)\,\de z}_{{}=:I_1}\\
 &- \underbrace{\int_{-1}^1\left(\mathcal{L}(\theta(z,t))-P_N \mathcal{L}(\theta^N(z,t))\right)\left(P_N \theta(z,t)- \theta^N(z,t)\right)w(z)\,\de z}_{{}=:I_2}\\
&- \underbrace{\int_{-1}^1\left(S(z,t)-P_N S(z,t)\right)\left(P_N \theta(z,t) - \theta^N(z,t)\right)w(z)\,\de z}_{{}=:I_3}.
\end{split}
\end{equation}
The orthogonal condition~\eqref{eq:orthogonal} implies that
\[
\int_{-1}^1\left(\partial_t\theta(z,t) - P_N \partial_t\theta(z,t)\right)\left(P_N \theta(z,t)-\theta^N(z,t)\right)w(z)\,\de z = 0,
\]
and
\[
\int_{-1}^1\left(S(z)-P_N S(z)\right)\left(P_N \theta(z,t) - \theta^N(z,t)\right)w(z)\,\de z =0.
\]
Thus,
\begin{equation}
\label{eq:I1}
\begin{split}
I_1 &= \int_{-1}^1\left(\partial_t\theta(z,t) - P_N \partial_t\theta(z,t)\right)\left(P_N \theta(z,t)-\theta^N(z,t)\right)w(z)\,\de z\\
& + \int_{-1}^1 \left(P_N \partial_t\theta(z,t)-\theta^N_{t}(z,t)\right)\left(P_N \theta(z,t)-\theta^N(z,t)\right)w(z)\,\de z\\
&= \frac{1}{2}\der{}{t}\norm{(P_N \theta - \theta^N)(\cdot,t)}^2_{H^1_w([-1,1])}.
\end{split}
\end{equation}
Since it is straightforward to see that $I_3 =0$, we can focus on $I_2$. From \eqref{eq:orthogonal} it follows that
\[
\int_{-1}^1\left(\mathcal{L}(\theta^N(z,t))-P_N \mathcal{L}(\theta^N(z,t))\right)\left(P_N \theta(z,t)-\theta^N(z,t)\right)w(z)\,\de z =0,
\]
and so, due to the locally Lipschitzianity of $K$ and $h_m$ and from Cauchy and triangular inequalities, we obtain
\begin{equation}
\label{eq:I2}
\begin{split}
I_2 &=\int_{-1}^1 \left(\mathcal{L}(\theta(z,t))-\mathcal{L}(\theta^N(z,t))\right)\left(P_N \theta(z,t)-\theta^N(z,t)\right)w(z)\,\de z\\
&\le C \norm{\left(\theta-\theta^N\right)(\cdot,t)}^2_{H^1_w([-1,1])} + C \norm{(P_N \theta - \theta^N)(\cdot,t)}^2_{H^1_w([-1,1])}\\
&\le 2C \norm{\left(\theta-P_N\theta\right)(\cdot,t)}^2_{H^1_w([-1,1])} + 3C \norm{(P_N \theta - \theta^N)(\cdot,t)}^2_{H^1_w([-1,1])}.
\end{split}
\end{equation}

Plugging~\eqref{eq:I1} and~\eqref{eq:I2} in~\eqref{eq:difference}, we have
\begin{equation}
\label{eq:substitution}
\der{}{t}\norm{(P_N \theta - \theta^N)(\cdot,t)}^2_{H^1_w([-1,1])} \le 4C \norm{(\theta-P_N\theta)(\cdot,t)}^2_{H^1_w([-1,1])} +6 C \norm{(P_N \theta - \theta^N)(\cdot,t)}^2_{H^1_w([-1,1])}.
\end{equation}

Since $\norm{(P_N \theta - \theta^N)(\cdot,0)}_{H^1_w([-1,1])} = 0$, Lemma~\ref{lm:sobolev} and Gronwall's inequality imply that
\begin{align*}
\norm{(P_N \theta - \theta^N)(\cdot,t)}^2_{H^1([-1,1])}&\le4C\int_0^t e^{6C(t-\tau)}\norm{(\theta-P_N \theta)(\cdot,\tau)}^2_{H^1_w([-1,1])}\,\de\tau\\
&\le \frac{C(T)}{N^{2s-2}}\int_0^t\norm{\theta(\cdot,\tau)}^2_{H^1_w([-1,1])}\,\de\tau.
\end{align*}

Hence,
\begin{equation}
\label{eq:2term}
\norm{P_N \theta - \theta^N}_{X_1}\le \frac{C(T)}{N^{s-1}}\norm{\theta}_{X_s}.
\end{equation}
Finally, using~\eqref{eq:1term} and~\eqref{eq:2term} in~\eqref{eq:triangular}, we complete the proof.
\end{proof}

\section{Numerical Simulations}\label{sec:numericalSimulations}

In this section we test our proposed method on different soils, possibly with a sink forcing term, representing the water uptake due to root systems. Van Genuchten - Mualem constitutive relations are considered in the following numerical simulations.

\begin{exm}\label{exm:1}
Drawing from \cite{Berardi_Difonzo_Notarnicola_Vurro_APNUM_2019}, we consider a soil with the following parameters:
\[
\theta_r=0.075,\,\theta_S=0.287,\,\alpha=0.036,\,n=1.56,\,K_S=0.94e-3\,\textrm{cm/s}.
\]
\begin{figure}
    \centering    \includegraphics[width=0.8\textwidth]{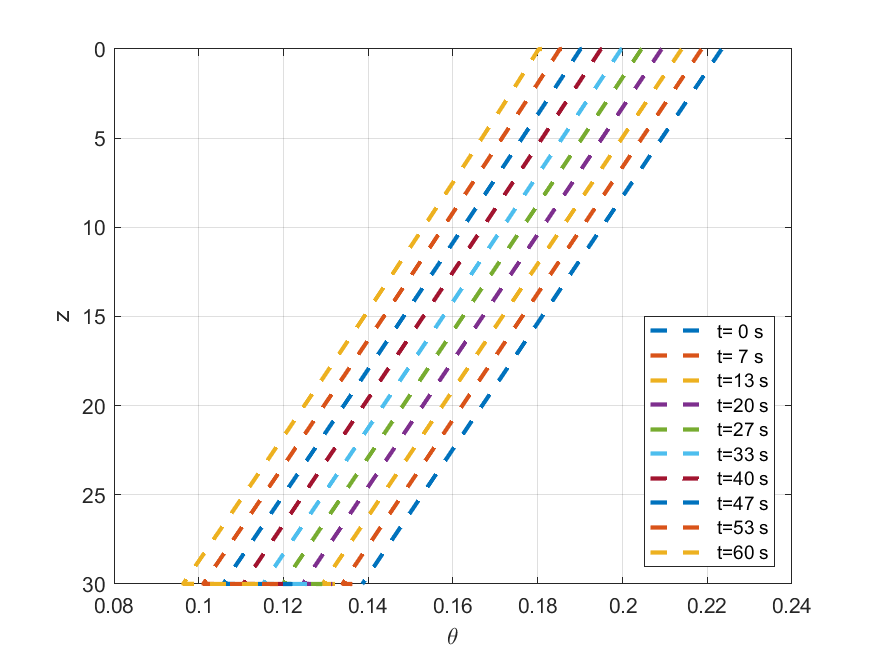}
    \caption{Numerical solution relative to Example \ref{exm:1}.}
    \label{fig:1}
\end{figure}
We added a sink term $S=-700\,\textrm{s}^{-1}$ and parameter $\delta=0.15$ in \eqref{eq:horizon}. We set our initial and boundary conditions as follows
\begin{align*}
\theta(0,t) &= 0.2234\left(1-\frac{t}{T}\right)+0.181\frac{t}{T},\,\,t\in[0,T], \\
\theta(Z,t) &= 0.1368\left(1-\frac{t}{T}\right)+0.1174\frac{t}{T},\,\,t\in[0,T], \\
\theta(z,0) &= 0.2234-\left(1-\frac{z}{Z}\right)\frac{0.0848}{2},\,\,z\in[0,Z].
\end{align*}
Here $Z=30$ cm, $T=60$ s; moreover, we used $\Delta t=0.06$ s and $\Delta x=0.3$ cm. Results are shown in Figure \ref{fig:1}.
\end{exm}

\begin{exm}\label{exm:3}
As already considered by \cite{Hills_et_al_1989}, we select a Glendale clay loam, characterized by the following  parameters
\[
\theta_r=0.1060,\,\theta_S=0.4686,\,\alpha=0.0104,\,n=1.3954,\,K_S=1.5162e-4\,\textrm{cm/s}.
\]
\begin{figure}
    \centering    \includegraphics[width=0.8\textwidth]{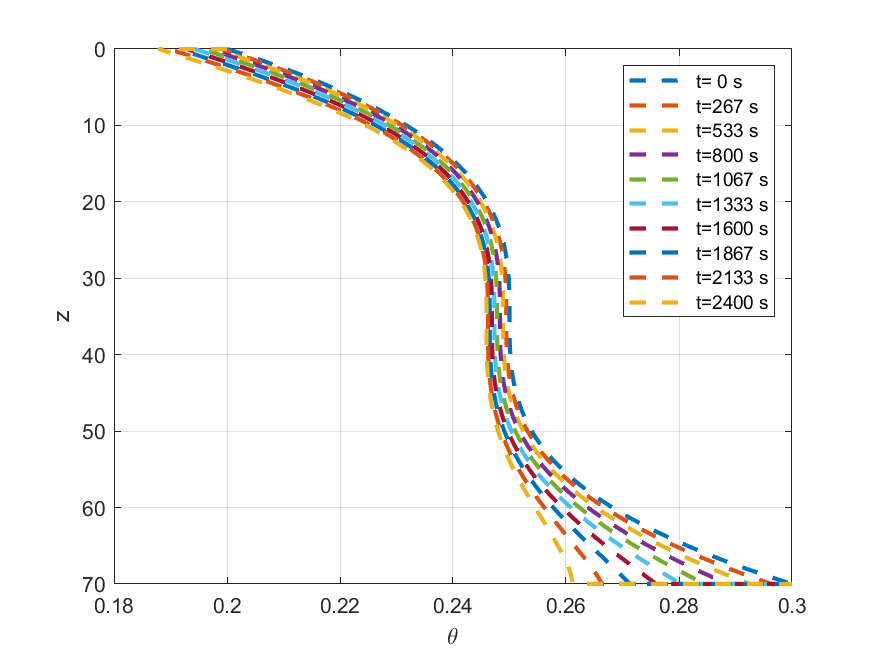}
    \caption{Numerical solution relative to Example \ref{exm:3}.}
    \label{fig:3}
\end{figure}
We put neither sink nor source on this simulation. Peridynamic parameter is $\delta=0.15$ in \eqref{eq:horizon}. Our boundary conditions are constant with values
\begin{align*}
\theta(0,t) &= 0.2,\,\,t\in[0,T], \\
\theta(Z,t) &= 0.3,\,\,t\in[0,T],
\end{align*}
while initial condition follows a nonlinear profile of the form
\[
\theta(z,0)=-0.05z^3+0.25,\,\,z\in[0,Z].
\]
We select $Z=70$ cm, $T=2400$ s; moreover, we used $\Delta t=2.4$ s and $\Delta x=0.3$ cm. The resulting water content profiles are shown in Figure \ref{fig:3}.
\end{exm}

\begin{exm}\label{exm:5}
As in \cite{Hills_et_al_1989}, we consider a Berino loamy fine sand, with parameters
\[
\theta_r=0.0286,\,\theta_S=0.3658,\,\alpha=0.0280,\,n=2.2390,\,K_S=0.0063\,\textrm{cm/s}.
\]
\begin{figure}
    \centering    \includegraphics[width=0.8\textwidth]{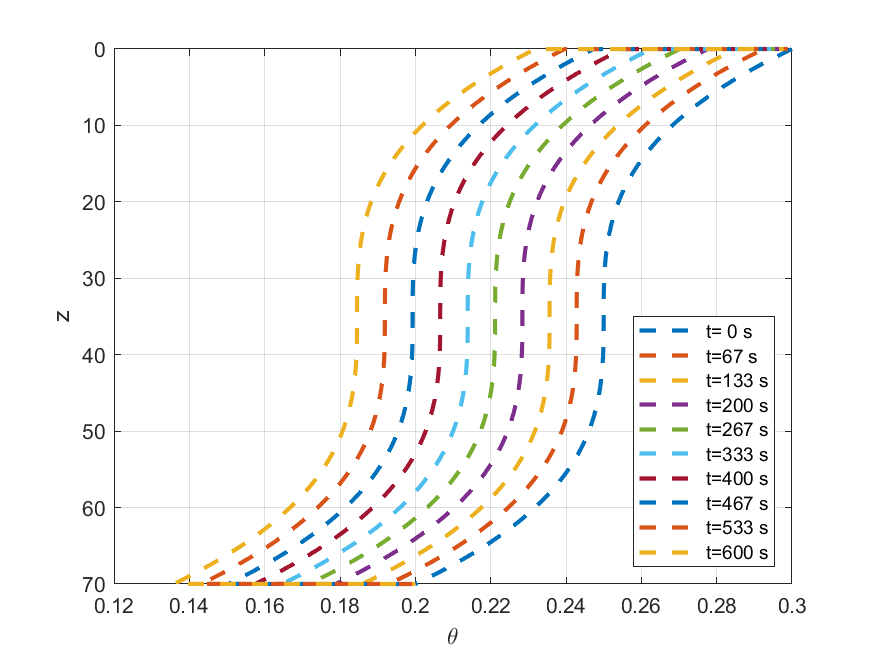}
    \caption{Numerical solution relative to Example \ref{exm:5}.}
    \label{fig:5}
\end{figure}
We added a sink term $S=-100\,\textrm{s}^{-1}$ and parameter $\delta=0.15$ in \eqref{eq:horizon}. We set our initial and boundary conditions as follows
\begin{align*}
\theta(0,t) &= 0.3\left(1-\frac{t}{T}\right)+0.29\frac{t}{T},\,\,t\in[0,T], \\
\theta(Z,t) &= 0.2,\,\,t\in[0,T],
\end{align*}
while initial condition has a nonlinear profile
\[
\theta(z,0)=0.05z^3+0.25,\,\,z\in[0,Z].
\]
We select $Z=70$ cm, $T=600$ s; moreover, we used $\Delta t=0.06$ s and $\Delta x=0.3$ cm. Results are shown in Figure \ref{fig:5}.
\end{exm}

\section{Conclusions and future works}\label{sec:conclusions}

Starting from an appropriate and physically based rewriting of Richards' equation using non-locality theory, we propose to compute its numerical solution using a semi-discretized time forward scheme based on Chebyshev transform. We prove that such approach converges in suitable Sobolev spaces, providing a theoretical background for further extensions of the present work to higher dimensional domains.
We also propose a new kind of convolutional kernel, or influence function, in order to manage the peridynamic behavior of the proposed model. Such an influence function distributes its effect on the domain so to correctly catch boundary conditions. In fact, we experienced major benefits from this choice, as numerical convergence turns out to be robust with respect to time, and compared to results coming from classical choices of influence functions. To testify our theoretical analysis, we have performed several experiments, on a wide range on soils, in \texttt{MATLAB}. We have considered different Dirichlet boundary conditions and linear and non-linear initial conditions and show that, with suitable discretization step-sizes, our method is reliable and accurate.

The present work suggests several possible directions for future and already ongoing research studies. For instance, it would be of interest applying Eulerian-Lagrangian methods (e.g. \cite{Abreu_et_al_JDDE_2022}) in the proposed peridynamic framework for Richards' equation, while the idea of introducing a basic control approach on the boundary conditions, as in \cite{Berardi_et_al_TiPM_2022}, could be adapted as well. Another development would consider non-local terms in time, so to resort to numerical solvers coming from specific tools in fractional differential calculus (see \cite{GarrappaPopolizio2011,DifonzoGarrappa2022}), or, even more promisingly, by spectral methods in 2D (see \cite{LPcheby2022,Bobaru2021,LP2021}).

\section*{Acknowledgments}
All authors are member of the INdAM Research group GNCS. MB and FVD also acknowledge the partial support of the 2022 project ``Modelli di evoluzione non locali: analisi, trattamento numerico e algoritmi" funded by GNCS-INdAM.
SFP acknowledges the partial support of ``Finanziamento giovani ricercatori 2022'' funded by GNCS-INdAM.
MB acknowledges the partial support of the CNR project ``MENTOR''. FVD has been supported by \textit{REFIN} Project, grant number 812E4967 funded by Regione Puglia. SFP has been supported by \textit{REFIN} Project, grant number D1AB726C funded by Regione Puglia.

\bibliographystyle{plain}

\bibliography{biblio}

\end{document}